\documentclass{amsart}

\usepackage{amssymb,amsmath,amscd}

\newtheorem{theorem}{Theorem}[section]                    
\newtheorem{proposition}[theorem]{Proposition}            
\newtheorem{corollary}[theorem]{Corollary}                
\newtheorem{lemma}[theorem]{Lemma}
\newtheorem{remark}[theorem]{Remark} 

\newtheorem{question}[theorem]{Question}

\newtheorem{claim}[theorem]{Claim}

\newtheorem{definition}{Definition}[section]

\newcommand{\act}{\curvearrowright}
\newcommand{\co}{\overline{\rm co}^{\rm w}}
\newcommand{\eci}{{\rm Esscom}_\infty}
\newcommand{\ecp}{{\rm Esscom}_p}
\newcommand{\efp}{{\rm Essfix}_p}
\newcommand{\efi}{{\rm Essfix}_\infty}
\newcommand{\Z}{\mathbb{Z}}
\newcommand{\Zp}{\mathbb{Z}_{+}}
\newcommand{\crpr}[3]{#1 \bar{\times}_{#2}#3}
\newcommand{\lcrpr}[3]{#1 \bar{\ltimes}_{#2}#3}
\newcommand{\rcrpr}[3]{#1 \bar{\rtimes}_{#2}#3}

\begin{document}
\title[Essential commutants of semicrossed products]{Essential commutants of semicrossed products}
\author[K.~Hasegawa]{Kei Hasegawa}
\address{Graduate~School~of~Mathematics, Kyushu~University, Fukuoka~819-0395, Japan}
\email{ma213034@math.kyushu-u.ac.jp}

\begin{abstract}
Let $\alpha:G\act M$ be a spatial action of countable abelian group on a ``spatial'' von Neumann algebra $M$
and $S$ be its unital subsemigroup with $G=S^{-1}S$.
We explicitly compute the essential commutant and the essential fixed-points, modulo the Schatten $p$-class or the compact operators, of the w$^*$-semicrossed product of $M$ by $S$ when $M'$ contains no non-zero compact operators.
We also prove a weaker result when $M$ is a von Neumann algebra on a finite dimensional Hilbert space and $(G,S)=(\Z,\Zp)$,
which extends a famous result due to Davidson (1977) for the classical analytic Toeplitz operators.
\end{abstract}

\maketitle
\section{Introduction}
Let $A$ be a (not necessarily self-adjoint) operator algebra on a Hilbert space $\mathcal{H}$.
We denote by $\mathfrak{S}_p=\mathfrak{S}_p(\mathcal{H})$ the Schatten $p$-class operators on $\mathcal{H}$ with $1\leq p < \infty$ and by $\mathfrak{S}_\infty=\mathfrak{S}_\infty(\mathcal{H})$ the compact operators $\mathcal{K}=\mathcal{K(H)}$ on $\mathcal{H}$.
We also denote by $\mathcal{I}(A)$ the set of all isometries in $A$.
In this paper we investigate the following two sets:
\begin{align*}
&\ecp (A)=\{ X\in\mathcal{B(H)} \,|\, aX - Xa \in \mathfrak{S}_p(\mathcal{H}) \ {\rm for}\  a\in A\},\\
&\efp (A)=\{ X\in\mathcal{B(H)} \,|\, v^*Xv-X \in \mathfrak{S}_p(\mathcal{H}) \ {\rm for}\  v\in \mathcal{I}(A) \},
\end{align*}
called the {\it essential commutant} and the {\it essential fixed-points} of $A$ modulo
the $*$-ideal $\mathfrak{S}_p$, respectively.
Clearly, $\ecp(A)$ is contained in $\efp(A)$,
and these two sets coincide when $A$ is a C$^*$-algebra which contains the identity operator,
since any unital C$^*$-algebra is the linear span of its unitary elements.
Johnson and Parrott \cite{Jo-Pa} and Popa \cite{Po} proved that
$\eci(A)=A' + \mathcal{K}$ holds when $A$ is a von Neumann algebra,
and Hoover \cite{Ho} proved that $\ecp(A)=A'+\mathfrak{S}_p$ when $A$ is a C$^*$-algebra but $p\neq \infty$.
On the other hand,
for a non-self-adjoint algebra these two sets do not coincide in general,
and thus the computation of them is non-trivial.
Such non-trivial works, among others, are $\eci(T(H^\infty))=T(H^\infty + C)+\mathcal{K}(H^2)$
due to Davidson \cite{Da} and $\efi(T(H^\infty))=T(L^\infty)+\mathcal{K}(H^2)$ due to Xia \cite{Xi},
where $T(H^\infty)$, $T(H^\infty + C)$, and $T(L^\infty)$ are the sets of all Toeplitz operators
on the Hardy space $H^2=H^2(\mathbb{T})$ whose symbols are in the bounded Hardy space $H^\infty=H^\infty(\mathbb{T})$, the Douglas algebra $H^\infty + C$ (with $C=C(\mathbb{T})$), and $L^\infty=L^\infty(\mathbb{T})$, respectively.
See \cite[Chapter 6,7]{Do} for those terminologies.

In this paper, we study the essential commutant and the essential fixed-points for w$^*$-semicrossed products.
The notation below follows \cite{Ka}.
Let $M$ be a von Neumann algebra on a Hilbert space $\mathcal{H}$,
$G$ be a countable abelian group acting on $M$ by $\alpha$,
and $S$ be a subsemigroup of $G$ which contains the unit of $G$ and generates $G$.
We also assume that the action $\alpha:G\act M$ is {\it spatial}, that is, $\alpha$ is implemented by unitary operators $\{u_g\,|\, g\in G\}$ on $\mathcal{H}$.
We denote by the same symbol $\alpha$ the action of $G$ on $M'$ implemented by $\{u_g \,|\, g\in G \}$.
Then we can construct the left and the right w$^*$-crossed products $\lcrpr{G}{\alpha}{M}$ and $\rcrpr{M'}{\alpha}{G}$
on $\mathbb{L}^2:=\mathcal{H}\otimes \ell^2(G)$ with $(\lcrpr{G}{\alpha}{M})'=\rcrpr{M'}{\alpha}{G}$.
The left and the right reduced w$^*$-semicrossed products $\lcrpr{S}{\alpha}{M}$ and $\rcrpr{M}{\alpha}{S}$ are constructed as $\sigma$-weakly closed subalgebras of $\lcrpr{G}{\alpha}{M}$ and $\rcrpr{M}{\alpha}{G}$, respectively.
Let $P$ denote the orthogonal projection from $\mathbb{L}^2$ onto $\mathbb{H}^2:=\mathcal{H}\otimes \ell^2(S)$
and define the Toeplitz map $T:\mathcal{B}(\mathbb{L}^2)\to \mathcal{B}(\mathbb{H}^2)$ by
$T_X:= PX|_{\mathbb{H}^2}$ for $X\in \mathcal{B}(\mathbb{L}^2)$.
The left and the right w$^*$-semicrossed products $\crpr{S}{\alpha}{M}$ and $\crpr{M}{\alpha}{S}$
are constructed as $\sigma$-weakly closed subalgebras of $\mathcal{B}(\mathbb{H}^2)$,
and in fact, coincide with $T(\lcrpr{S}{\alpha}{M})$ and $T(\rcrpr{M}{\alpha}{S})$, respectively (see Proposition \ref{BH}),
where for a subset $\mathfrak{F}$ of $\mathcal{B}(\mathbb{L}^2)$ the $T(\mathfrak{F})$ stands for $\{T_F\,|\,F\in \mathfrak{F} \}$.
In the typical case of $(M,\mathcal{H},G,S)=(\mathbb{C},\mathbb{C},\Z,\Zp)$,
the ``left'' and the ``right'' algebras $(\lcrpr{G}{\alpha}{M},\lcrpr{S}{\alpha}{M},\crpr{S}{\alpha}{M})$ and $(\rcrpr{M}{\alpha}{G},\rcrpr{M}{\alpha}{S},\crpr{M}{\alpha}{S})$ coincide and become $(L^\infty, H^\infty,T(H^\infty))$.
The concept of w$^*$-semicrossed products and reduced ones were introduced by Kakariadis \cite{Ka} for $\sigma$-weakly closed operator algebras and their normal endomorphisms
and for von Neumann algebras and their automorphisms, i.e., the case of $(G,S)=(\Z,\Zp)$, respectively.
The latter coincides with the adjoint of  analytic crossed products studied by McAsey, Muhly, and Saito \cite{Mc-Mu-Sa}.
Toeplitz operators associated with analytic crossed products were studied by Saito \cite{Sa} and
the algebras of analytic Toeplitz operators in this sense essentially coincide with w$^*$-semicrossed products as above.

In Section 2 we define these objects and prove some elementary properties.
When $M'$ contains no non-zero compact operators,
we can explicitly compute $\ecp( \crpr{S}{\alpha}{M})$ and $\efp(\crpr{S}{\alpha}{M})$ as follows. 

\begin{theorem}\label{thmd}
Let $M\subset \mathcal{B(H)}$ be a von Neumann algebra, $\alpha :G \act M$ a spatial action of a discrete countable abelian group, and $S$ be a subsemigroup of $G$ which contains the unit of $G$ and generates $G$.
If $M'$ contains no non-zero compact operators, then
\begin{align*}
\efp( \crpr{S}{\alpha}{M})=T(\rcrpr{M'}{\alpha}{G}) + \mathfrak{S}_p(\mathbb{H}^2),\;\;
\ecp ( \crpr{S}{\alpha}{M})=\crpr{M'}{\alpha}{S} + \mathfrak{S}_p(\mathbb{H}^2)
\end{align*}
hold for every $1 \leq p \leq \infty$.
\end{theorem}

\begin{corollary}\label{cor}
The same assertion of Theorem \ref{thmd} holds when $\mathcal{H}$ is the standard Hilbert space $L^2(M)$
and $M$ is either diffuse, $\mathcal{B}(\ell^2)$ of infinite dimension, or a {\rm (}possibly infinite{\rm )} direct sum of
them.
\end{corollary}

An another immediate corollary of Theorem \ref{thmd} is the double commutant theorem for $\crpr{S}{\alpha}{M}$ and $\crpr{S}{\alpha}{M'}$ in the Calkin algebra under the assumption that $M\cap \mathcal{K}=M'\cap\mathcal{K}=\{0\}$.
The proofs of Theorem \ref{thmd} and Corollary \ref{cor} will be given in Section 3.
In Section 4, we consider the case that $\mathcal{H}$ is finite dimensional (and hence so is $M$) and $(G,S)=(\Z, \Zp)$,
and can prove the following theorem:

\begin{theorem}\label{thmf}
Let $M$ be a von Neumann algebra on a finite dimensional Hilbert space $\mathcal{H}$
and $\alpha$ be a $*$-automorphism on $M$ implemented by a unitary operator on $\mathcal{H}$.
Then we have
\begin{align*}
\ecp(\crpr{\Zp}{\alpha}{M}) \subset T(\rcrpr{M'}{\alpha}{\Z}) + \mathfrak{S}_p(\mathbb{H}^2) \quad {\rm for}\;\; 1\leq p \leq \infty.
\end{align*}
\end{theorem}
This theorem is nothing but \cite[Theorem 2]{Da} for the classical analytic Toeplitz operators
by specializing $M=\mathbb{C}$ and $p=\infty$,
though our proof is a bit improved with the help of several ideas in \cite{Mu-Xi},\cite{Xi}. 
The reason why we do not consider the essential fixed-points is that there is a technical obstruction to translating the argument in \cite{Xi} into our setting (see Remark \ref{remxia}).
In the final section, we examine whether or not the assertion of Theorem \ref{thmf} still holds when $M$ is an arbitrary von Neumann algebra of standard form.

%
%

\section{Preliminaries}
We begin with a small remark.
\begin{remark}\normalfont\label{rem1}
For a given $v \in \mathcal{I}(A)$ we consider the unital completely positive map $\Psi_v :X \mapsto v^*Xv$ on $\mathcal{B(H)}$. 
Then $\efp (A)$ is nothing but the set of operators ``fixed  modulo $\mathfrak{S}_p$'' by $\{ \Psi_v \,|\, v \in \mathcal{I}(A) \}$.
It is immediate to see that $\ecp (A) \subset \efp (A)$.
Moreover, these two sets coincide if $A$ is a C$^*$-algebra, since any C$^*$-algebra is generated by its unitary elements.
By \cite{Jo-Pa}, \cite{Po} and \cite{Ho} any von Neumann algebra $M\subset \mathcal{B(H)}$ enjoys $\ecp(M)=\efp(M)=M'+\mathfrak{S}_p$ for every $1\leq p \leq \infty$.
\end{remark}

Next, let us recall w$^*$-semicrossed products and Toeplitz maps associated with them.
Our notation and formulation follow \cite{Ka}.
Remark that those do not completely agree with the usual ones for crossed products.
Let $M\subset \mathcal{B(H)}$ be a von Neumann algebra and $G$ be a countable discrete abelian group.
Assume that we have a spatial action $\alpha : G \act M$ implemented by 
a unitary representation $u:G \ni g \mapsto u_g \in \mathcal{B(H)}$,
i.e., we have $\alpha_g(x)=u_gxu^*_g$ for every $x\in M$ and $g\in G$.
Since $G$ acts on $M'$ too by ${\rm Ad}\,u_g$,
we also denote by the same symbol $\alpha_g$ the automorphism ${\rm Ad}\,u_g$ on $M'$.
Let $G\ni g \mapsto \lambda_g \in \mathcal{B}(\ell^2(G))$ be the left regular representation of $G$
and $e_g \in \mathcal{B}(\ell^2(G))$ be the orthogonal projection onto $\mathbb{C}\delta_g$.
Set $\mathbb{L}^2:=\mathcal{H}\otimes \ell^2(G)$
and define representations $\pi : M \to \mathcal{B}(\mathbb{L}^2)$ and $\lambda(\cdot), \rho(\cdot) :G\to \mathcal{B}(\mathbb{L}^2)$ by
$$
\pi(x)=\sum_{g\in G}\alpha_g(x)\otimes e_g,\quad
\lambda(g)=1\otimes \lambda_g,\quad
\rho(g)=u^*_g\otimes \lambda_g
$$
for $x\in M$ and $g\in G$.
The {\it left} and the {\it right w$^*$-crossed products} of $M$ by $G$ are defined to be
the von Neumann algebras $\lcrpr{G}{\alpha}{M}=\pi(M)\vee \{ \lambda(g) \,|\, g\in G \}''$ and $\rcrpr{M}{\alpha}{G}:=M\otimes \mathbb{C}1\vee \{ \rho(g) \,|\, g\in G \}''$, respectively.
It is well-known that $(\lcrpr{G}{\alpha}{M})'=\rcrpr{M'}{\alpha}{G}$ (see e.g. \cite[Theorem 1.21]{Ta}).
We also note that when $M$ is of standard form, each action $\alpha :G\act M$ is spatial thanks to \cite[Theorem 3.2]{Ha}.

\begin{definition}[Kakariadis \cite{Ka}]\normalfont\label{defsc}
Let $M\subset \mathcal{B(H)}$ and $\alpha:G\act M$ be as above.
For a given subsemigroup $S$ of $G$ which contains the unit $\iota$ of $G$ and generates $G$,
the {\it left} and the {\it right reduced w$^*$-semicrossed product} $\lcrpr{S}{\alpha}{M}$ and $\rcrpr{M}{\alpha}{S}$ of $M$ by $S$
are defined to be the $\sigma$-weakly closed subalgebras of $\lcrpr{G}{\alpha}{M}$ and $\rcrpr{M}{\alpha}{G}$
generated by $\pi(M)$ and $\{ \lambda(s) \,|\, s\in S \}$ and by $M\otimes \mathbb{C}1$ and $\{ \rho(s) \,|\, s\in S \}$, respectively.

Let $P$ be the projection onto $\mathbb{H}^2:=\mathcal{H}\otimes \ell^2(S) \subset \mathbb{L}^2$.
The {\it Toeplitz map} $T:\mathcal{B}(\mathbb{L}^2)\to \mathcal{B}(\mathbb{H}^2)$ is defined by $T_X:=PX|_{\mathbb{H}^2}$ for $X \in \mathcal{B}(\mathbb{L}^2)$.
The {\it left} and the {\it right w$^*$-semicrossed product} $\crpr{S}{\alpha}{M}$ and $\crpr{M}{\alpha}{S}$ of $M$
by $S$ are defined to be
$\sigma$-weakly closed subalgebras of $\mathcal{B}(\mathbb{H}^2)$ generated by $T(\pi(M))$ and $\{T_{\lambda(s)} \,|\,s\in S\}$ and by $M\otimes \mathbb{C}1_{\ell^2(S)}$ and $\{T_{\rho(s)} \,|\,s\in S\}$, respectively.
Here for a subset $\mathfrak{F}$ of $\mathcal{B}(\mathbb{L}^2)$ the $T(\mathfrak{F})$ stands for $\{T_F\,|\,F\in \mathfrak{F} \}$.
\end{definition}

\begin{remark}\normalfont\label{reminc}
Since $T$ is normal and multiplicative on $\lcrpr{S}{\alpha}{M}$ and $\rcrpr{M}{\alpha}{S}$,
one has $T(\lcrpr{S}{\alpha}{M}) \subset \crpr{S}{\alpha}{M}\subset \overline{T(\lcrpr{S}{\alpha}{M})}^{\sigma\mathchar`-{\rm w}}$
and $T(\rcrpr{M}{\alpha}{S}) \subset \crpr{M}{\alpha}{S}\subset \overline{T(\rcrpr{M}{\alpha}{S})}^{\sigma\mathchar`-{\rm w}}$.
Moreover, for any $x,z \in (\lcrpr{S}{\alpha}{M}) \cup (\rcrpr{M}{\alpha}{S})$ and $Y\in \mathcal{B}(\mathbb{L}^2)$, we have $T_{x}^*T_Y T_z=T_{x^*Yz}$.
We also note that $P$ is in $\pi(M)'$, and hence one has $$\efp(T(\pi(M)))=\ecp(T(\pi(M)))=T(\pi(M)')+\mathfrak{S}_p.$$
We will use these facts throughout.
\end{remark}

\begin{remark}\normalfont\label{remuni}
Define a unitary operator $W:=\sum_{g\in G} u_g\otimes e_g$ on $\mathbb{L}^2$.
It is easily seen that $W^* \pi (x) W=x\otimes 1$ and $W^*\lambda(g) W=\rho(g)$ for $x\in M$ and $g\in G$.
Moreover, since $W$ commutes with $P$, $\hat{W}:=T_W$ is also unitary on $\mathbb{H}^2$,
and hence one has
$\hat{W}^*T_{\pi(x)}\hat{W}=x\otimes 1_{\ell^2(S)}$ and $\hat{W}^*T_{\lambda(g)} \hat{W}=T_{\rho(g)}$
for $x\in X$ and $g\in G$.
Therefore, the ``left'' and the ``right'' algebras are unitarily equivalent.
\end{remark}

\begin{proposition}\label{prop}
Let $M,G,S,\alpha$ be as in Definition \ref{defsc}.
Then we have $\lcrpr{S}{\alpha}{M}=\{ x\in \lcrpr{G}{\alpha}{M}\,|\, x(1\otimes e_\iota)=Px(1\otimes e_\iota)\}$
and $\rcrpr{M}{\alpha}{S}=\{ x\in \rcrpr{M}{\alpha}{G}\,|\, x(1\otimes e_\iota)=Px(1\otimes e_\iota)\}$.
\end{proposition}

\begin{proof}
Clearly, $\lcrpr{S}{\alpha}{M}$ is contained in $\{ x\in \lcrpr{G}{\alpha}{M}\,|\, x(1\otimes e_\iota)=Px(1\otimes e_\iota)\}$.
Let $\hat{G}$ be the dual group of $G$, and $\hat{\alpha}:\hat{G}\to {\rm Aut}(\lcrpr{G}{\alpha}{M})$ the dual action.
By \cite[Corollary 4.3.2]{Lo-Mu},
$\lcrpr{S}{\alpha}{M}$ is the spectral subspace of $\lcrpr{G}{\alpha}{M}$ associated with $S$ by the dual action, 
and we denote it by $(\lcrpr{G}{\alpha}{M})^{\hat{\alpha}}(S)$.
It is easily seen that $(\lcrpr{G}{\alpha}{M})^{\hat{\alpha}}(S)$ contains
$\{ x\in \lcrpr{G}{\alpha}{M}\,|\, x(1\otimes e_\iota)=Px(1\otimes e_\iota)\}$.
By the preceding remark, one has $\rcrpr{M}{\alpha}{S}=\{ x\in \rcrpr{M}{\alpha}{G}\,|\, x(1\otimes e_\iota)=Px(1\otimes e_\iota)\}$.
\end{proof}

Recall that a semigroup $S$ is {\it right amenable} if $S$ has a {\it right invariant mean},
that is, there exists a state $\psi$ on $\ell^\infty(S)$ which satisfies that $\psi(f)=\psi(r_sf)$ for $f\in \ell^\infty(S)$ and $s\in S$,
where $r_sf(t) = f(ts), \, t\in S$.
It is known, see \cite[Theorem 17.5]{He-Ro}, that every abelian semigroup is (right) amenable.

\begin{proposition}\label{lemce}
If $S$ is a right amenable semigroup and $\sigma:S\rightarrow \mathcal{B(H)}$ is a unitary representation,
then $\{\sigma(s) \,|\, s\in S\}' \cap \co \{\sigma(s)^*x\sigma(s) \,|\, s\in S \}\neq \emptyset$ for every $x\in \mathcal{B(H)}$. 
\end{proposition}
\begin{proof}
Let $\psi$ be a right invariant mean on $S$.
Fix $x \in \mathcal{B(H)}$.
For $\xi\in \mathfrak{S}_1(\mathcal{H})$, define $f_{\xi} \in \ell^\infty(S)$
by $f_{\xi}(s)={\rm Tr}(\sigma(s)^*x\sigma(s) \xi), s\in S$.
Then there exists $y\in \mathcal{B(H)} \cong \mathfrak{S}_1(\mathcal{H})^*$ such that
${\rm Tr}(y\xi)=\psi(f_{\xi})$ for $\xi \in \mathfrak{S}_1(\mathcal{H})$.
Since $f_{\sigma(s) \xi \sigma(s)^*}= r_s f_\xi$ for $s\in S$
and $\psi$ is right invariant, we have
$
{\rm Tr}(\sigma(s)^* y \sigma(s) \xi)=\psi(f_{\sigma(s) \xi \sigma(s)^*}) = \psi(f_{\xi})={\rm Tr}(y \xi),
$
which implies that $y\in \{\sigma(s) \,|\, s\in S \}'$.
Suppose that $ y \notin \co \{\sigma(s)^*x\sigma(s) \,|\, s\in S \} $.
By the Hahn-Banach separation theorem, there exist $\xi \in \mathfrak{S}_1(\mathcal{H})$ and a constant $c \in \mathbb{R}$
such that ${\rm Re \,} \psi(f_\xi) = {\rm Re\,Tr}(y\xi)< c \leq  {\rm Re\,Tr}(\sigma(s)^*x\sigma(s))= {\rm Re \,} f_\xi(s) $ for $s\in S$.
However, by the Kre\v{\i}n-Mil'man theorem, $\psi$ falls in the weak$^*$ closed convex hull of $S \subset \ell^\infty(S)^*$, a contradiction.
\end{proof}

The following proposition gives us a Brown-Halmos type criterion (\cite[Theorem 6 and 7]{Br-Ha}).
\begin{proposition}\label{BH}
Let $M\subset \mathcal{B(H)},S\subset G,\alpha$ be as in Definition \ref{defsc}.
Then the following are true.
\begin{itemize}
\item[(i)] For a given $X\in \mathcal{B}(\mathbb{H}^2)$, $X$ falls in $T(\rcrpr{M'}{\alpha}{G})$ $($resp. $T(\lcrpr{G}{\alpha}{M'}))$ if and only if
$X$ commutes with $T(\pi(M))$ $($resp. $M\otimes \mathbb{C}1_{\ell^2(S)})$ and satisfies that $T_{\lambda(s)}^*XT_{\lambda(s)}=X$ $($resp. $T_{\rho(s)}^*XT_{\rho(s)}=X)$ for every $s\in S$.
\item[(ii)] $T(\rcrpr{M'}{\alpha}{S})= \crpr{M'}{\alpha}{S}=(\crpr{S}{\alpha}{M})'$ and $T(\lcrpr{S}{\alpha}{M'}) = \crpr{S}{\alpha}{M'}=(\crpr{M}{\alpha}{S})'$.
\end{itemize}
\end{proposition}
\begin{proof}
Let $X\in T(\pi(M))'$ be arbitrarily chosen in such a way that $T_{\lambda(s)}^*XT_{\lambda(s)}=X$ for every $s\in S$.
Since $T(\pi(M))'=T(\pi(M)')$, there exists $x\in \pi(M)'$ such that $X=T_x$.
By Proposition \ref{lemce}, we find $y$ in $\co \{ \lambda(s)^* PxP \lambda(s) \,|\, s\in S \} \cap \{\lambda(s) \,|\, s\in S \}'$.
Note that $y$ is in $\rcrpr{M'}{\alpha}{G}=(\lcrpr{G}{\alpha}{M})'=\pi(M)'\cap \{\lambda(g)\,|\, g\in G\}'$
since $PxP$ is in $\pi(M)'$, the $\lambda(s), s\in S$, normalize $\pi(M)'$, and $S$ generates $G$.
Since the Toeplitz map $T$ is $\sigma$-weakly continuous, one has
$
T_y \in \co \{T_{\lambda(s)}^* X T_{\lambda(s)} \,|\, s\in S \} = \{X\},
$
which implies that $X=T_y \in T(\rcrpr{M'}{\alpha}{G})$.
Conversely, let $x$ be in $\rcrpr{M'}{\alpha}{G}\subset \pi(M)'$.
Then $T_x\in T(\pi(M)')=T(\pi(M))'$.
For any $s\in S$, we have $T_{\lambda(s)}^*T_xT_{\lambda(s)}=T_{\lambda(s)^*x\lambda(s)}=T_x$,
which implies (i).

To see (ii) it suffices to prove that $(\crpr{S}{\alpha}{M})'= T(\rcrpr{M'}{\alpha}{S})$ by Remark \ref{reminc}.
It is immediate to see that $T(\rcrpr{M'}{\alpha}{S}) \subset (\crpr{S}{\alpha}{M})' $.
Conversely, let $Y\in (\crpr{S}{\alpha}{M})'$ be arbitrary.
By the preceding paragraph, there exists $a \in \rcrpr{M'}{\alpha}{G}$ such that
$Y=T_a$. 
By the assumption that $G=S^{-1}S$,
for each $g\in G\setminus S$ there exist $s,t\in S$ such that $g=s^{-1}t$.
Since $a$ commutes with $\lambda(s), s\in S$, we have 
\begin{align*}
\langle a \xi\otimes \delta_\iota , \eta\otimes \delta_g \rangle 
&=\langle (1-P) a P \xi\otimes \delta_\iota, \eta\otimes \delta_{s^{-1}t} \rangle 
=\langle P \lambda(s) (1-P) a P\xi \otimes \delta_\iota, \eta \otimes \delta_t \rangle \\
&=\langle [T_a,T_{\lambda(s)}] \xi\otimes \delta_\iota, \eta\otimes \delta_t \rangle 
=0
\end{align*}
for every $\xi,\eta \in \mathcal{H}$,
and hence $(1\otimes e_g) a (1\otimes e_\iota) = 0$.
Therefore, $a$ falls in $\rcrpr{M'}{\alpha}{S}$ by Proposition \ref{prop}.
\end{proof}

%
%
\section{Proof of Theorem \ref{thmd}}
Throughout this section, let $M\subset \mathcal{B(H)}$ denote a von Neumann algebra, $\alpha:G\act N$ a spatial action of discrete countable abelian group, and $S$ a subsemigroup of $G$ which contains the unit $\iota$ of $G$ and generates $G$.

\begin{lemma}\label{lemnocpt}
If $M'$ $($resp. $M)$ contains no non-zero compact operators,
then so does $\pi(M)'$ $($resp. $(M'\otimes \mathbb{C}1)')$.
\end{lemma}

\begin{proof}
Assume that $M'\cap \mathcal{K}=\{0 \}$.
By Remark \ref{remuni} it suffices to prove that $(M\otimes \mathbb{C}1)'\cap \mathcal{K}(\mathbb{L}^2)=\{0\}$.
Let $K\in (M\otimes \mathbb{C}1)'\cap \mathcal{K}(\mathbb{L}^2)$ be arbitrary chosen.
Since $(M\otimes \mathbb{C}1)'=M'\bar{\otimes}\mathcal{B}(\ell^2(G))$,
the $(1\otimes \lambda^*_g e_g) K (1\otimes e_h\lambda_h)|_{\mathcal{H}\otimes \mathbb{C}\delta_\iota}$ falls in $(M'\otimes \mathbb{C}e_\iota) \cap \mathcal{K}(\mathcal{H}\otimes \mathbb{C}\delta_\iota) \cong M'\cap \mathcal{K(H)}=\{0 \}$ for every $g,h \in G$, which implies $K=0$.
\end{proof}

\begin{lemma}
The restrictions of the Toeplitz map to $\{\lambda(g) \,|\, g\in G \}'$ and $\{ \rho(h) \,|\, h\in G  \}'$ are isometries.
Consequently, every isometry in $\crpr{S}{\alpha}{M}$ and $\crpr{M}{\alpha}{S}$ is
of the form $T_v$ with some isometry $v$ in $\lcrpr{S}{\alpha}{M}$ and $\rcrpr{M}{\alpha}{S}$, respectively.
\end{lemma}

\begin{proof} 
Firstly, we prove that the restriction of $T$ to $\{\lambda(g) \,|\, g\in G  \}'$ is injective.
Let $x\in \{\lambda(g) \,|\, g\in G \}'$ be chosen in such a way that $T_x=0$.
Since $S$ generates $G$ and $G$ is abelian,
for each $g,h\in G$ there exist $s,s',t,t' \in S$ such that $g=t^{-1}s$ and $h=t'^{-1}s'$.
Since $x$ commutes with $\lambda(g),g\in G$,
one has 
\begin{align*}
\langle x(\xi\otimes \delta_{t^{-1}s}),\eta\otimes \delta_{t'^{-1}s'}\rangle 
&= \langle x (\xi\otimes \lambda_t^* \delta_{s}), \eta\otimes \lambda _{t'}^*\delta_{s'}\rangle 
= \langle x(\xi\otimes \lambda_{t'} \delta_{s}), \eta\otimes \lambda _{t}\delta_{s'}\rangle \\
&= \langle T_x (\xi\otimes \delta_{t's}), \eta\otimes \delta_{ts'} \rangle
= 0
\end{align*}
for every $\xi,\eta \in \mathcal{H}$, and hence $x=0$.
By Proposition \ref{lemce} there exists $y \in \co \{ \lambda(s)^* PxP \lambda(s) \,|\, s\in S \} \cap \{\lambda(g) \,|\, g\in G \}'$.
Note that $\| y\| \leq \| T_x\|$.
Since $P\lambda(s) P = \lambda(s) P$ for $s \in S$,
one has
$$
T_y \in \co\{ T_{\lambda(s)}^* T_x T_{\lambda(s)} \,|\, s\in S \} =  \co\{T_{\lambda(s)^*x\lambda(s)} \,|\, s\in S \} = \{T_x \}.
$$
Since $T$ is injective on $\{\lambda(g) \,|\, g\in G  \}'$, 
it follows that $x=y$.
Thus $\| x\| \leq \| T_x\| \leq \| x\|$, and hence the restriction of $T$ to $\{ \lambda(g) \,|\, g\in G\}' $ is isometric.

To see the second assertion let $V \in \mathcal{I}(\crpr{S}{\alpha}{M})$ be arbitrary.
By Proposition \ref{BH} we find $v\in \lcrpr{S}{\alpha}{M}$ in such a way that $V=T_v$.
Then one has $1=V^*V=T_{v}^*T_v=T_{v^*v}$. Since $T$ is injective on $\lcrpr{S}{\alpha}{M}\subset \{\rho(g)\,|\,g\in G\}'$, $v$ itself must be isometry. 
Similarly, one can prove the same assertion for $\{ \rho(g) \,|\, g\in G \}'$ and $\crpr{M}{\alpha}{S}$.
\end{proof}

By the preceding lemma, it is immediate to see that $T(\rcrpr{M'}{\alpha}{G}) + \mathfrak{S}_p$ is contained in
$\efp(\crpr{S}{\alpha}{M})$.
Hence the next theorem completes the proof of Theorem \ref{thmd}.

\begin{theorem}
Assume that $M'\cap \mathcal{K}=\{0\}$.
For a given $X\in \mathcal{B}(\mathbb{H}^2)$ the following are true.
\begin{itemize}
\item[(i)] The $X$ falls in $T(\rcrpr{M'}{\alpha}{G}) +\mathfrak{S}_p$ if and only if
$X$ commutes with $T(\pi(M))$ modulo $\mathfrak{S}_p$
and satisfies that $T_{\lambda(s)}^*XT_{\lambda(s)}-X \in \mathfrak{S}_p$ for every $s\in S$.
\item[(ii)] The $X$ falls in $\crpr{M'}{\alpha}{S}+\mathfrak{S}_p$ if and only if $X$ commutes with $T(\pi(M))$ and $\{T_{\lambda(s)} \,|\, s\in S\}$ modulo $\mathfrak{S}_p$.
\end{itemize}

\end{theorem}
\begin{proof}
First, we prove the `if' part of (i).
Let $X \in \ecp(T(\pi(M)))$ be arbitrarily chosen in such a way that $T_{\lambda(s)}^*XT_{\lambda(s)}-X \in \mathfrak{S}_p$ for every $s\in S$.
By Remark \ref{reminc} there exist $a\in\pi(M)'$ and $K\in \mathfrak{S}_p$ such that  $X=T_a+K$.
Note that $(T_a - T_{\lambda(s)}^*T_aT_{\lambda(s)})P=P(a-\lambda(s)^*a \lambda(s))P \in \pi(M)'$
since $a,P \in \pi(M)'$ and $\lambda(s),s\in S,$ normalize $\pi(M)'$.
For each $s\in S$ one has
$$
(T_a - T_{\lambda(s)}^*T_aT_{\lambda(s)})P
= (X- T_{\lambda(s)}^*X T_{\lambda(s)})P - (K - T_{\lambda(s)}^*K T_{\lambda(s)})P
$$
falls into $\mathcal{K}\cap \pi(M)'$ by assumption.
Thus, it follows from Lemma \ref{lemnocpt} that $T_{\lambda(s)}^*T_aT_{\lambda(s)}$ $=T_a$  for every $s\in S$, and hence we have $T_a \in T(\rcrpr{M'}{\alpha}{G})$ by Proposition \ref{BH}.
Therefore, the $X=T_a+K$ is in $T(\rcrpr{M'}{\alpha}{G}) + \mathfrak{S}_p$.
The `only if' part of (i) follows from Proposition \ref{BH} again.

Next, we prove (ii). 
Since $\crpr{M'}{\alpha}{S} = (\crpr{S}{\alpha}{M})'$, the `only if' part is trivial.
By the preceding paragraph, it suffices to prove that every $T_y\in T(\rcrpr{M'}{\alpha}{G})$ commuting with $\{T_{\lambda(s)} \,|\, s\in S\}$ modulo $\mathfrak{S}_p$ falls in $\crpr{M'}{\alpha}{S}=T(\rcrpr{M'}{\alpha}{S})$.
By Proposition \ref{prop},
it suffices to show that $(1\otimes e_g) y (1\otimes e_\iota) = 0$ for every $g\in G\setminus S$.
Note that $(1\otimes e_g) y (1\otimes e_\iota)\in \pi(M)'$ since $(1\otimes e_h)\in \pi(M)'$ for every $h\in G$.
As in the proof of Proposition \ref{BH} we can find $s,t\in S$ in such a way that $g=s^{-1}t$, 
and have
$\langle y \xi\otimes \delta_\iota , \eta\otimes \delta_g \rangle =\langle [T_y,T_{\lambda(s)}] \xi\otimes \delta_\iota, \eta\otimes \delta_t \rangle$
for every $\xi,\eta \in \mathcal{H}$,
which implies that $(1\otimes e_g) y (1\otimes e_\iota)$ is compact,
and hence $(1\otimes e_g)y(1\otimes e_\iota) \in \pi(M)'\cap \mathcal{K}=\{ 0 \}$.
Therefore, we get $T_y \in \crpr{M'}{\alpha}{S}$.
\end{proof}

\begin{proof}[Proof of Corollary \ref{cor}]
By \cite[Theorem 3.2]{Ha} every action $\alpha:G\act M$ is spatial.
Since $M$ is anti-spatially isomorphic to $M'$,
it suffices to prove that $\mathcal{K}\cap M =\{0\}$.
We first prove the case when $M=\mathcal{B}(\ell^2)$.
Since $(\mathcal{B}(\ell^2), L^2(\mathcal{B}(\ell^2)))\cong (\mathcal{B}(\ell^2) \otimes \mathbb{C}1, \ell^2\otimes \overline{\ell^2})$ and $\ell^2$ is infinite dimensional, we have $\mathcal{K}\cap (\mathcal{B}(\ell^2)\otimes \mathbb{C}1)=\{0\}$.
When $M$ is diffuse, it is clear that $\mathcal{K}\cap M=\{0\} $.
The general case follows from \cite[Lemma 2.6]{Ha}, which guarantees that each central projection $q \in M$ enjoys $(Mq, qL^2(M)) \cong (Mq, L^2(Mq))$.
\end{proof}

%
%
\section{Proof of Theorem \ref{thmf}}
Let $(M,\mathcal{H},\alpha)$ be as in Theorem $\ref{thmf}$.
We first point out that the same assertions of the lemmas below hold true for $\rho(1)$ since $\lambda(1)$ and $\rho(1)$ are unitarily equivalent, see Remark \ref{remuni}.
Remark that $\lambda(n)$ converge to 0 weakly,
and hence $\lambda(n)^* K \lambda(n)$ converge to 0 strongly for every compact operator $K$.
Also, note that $(1-P)\lambda(n)$ converges to 0 strongly.
These facts are frequently used throughout.

The following lemma seems well-known, but we do give its proof for the sake of completeness.
\begin{lemma}\label{lemab}
There exists a $*$-isomorphism from $L^\infty=L^\infty(\mathbb{T})$ onto $\{ \lambda(1) \}''$
sending $z$ to $\lambda(1)$.
\end{lemma}
\begin{proof}
Let $U$ be the bilateral shift on $L^2(\mathbb{T})$ with respect to the standard basis $\{z^n\,|\, n\in \Z \}$
and define the unitary transformation $V:\ell^2(\Z) \to L^2(\mathbb{T})$ by $V\delta_n:=z^n, n\in \Z$.
Then one has $\lambda(1)=1\otimes \lambda_1=1\otimes V^*U V$.
Since the von Neumann algebra generated by $U$ is known to be $L^\infty$,
the correspondence $L^\infty \ni f \mapsto 1\otimes V^* f V \in \{\lambda(1) \}''$ apparently gives the desired $*$-isomorphism.
\end{proof}

Recall that the {\it essential norm} $\|X\|_{\rm e}$ of $X \in \mathcal{B(H)}$ is defined to be $\| X \|_{\rm e} =\inf \{\| X-K \| \,|\, K \in \mathcal{K(H)} \}$
and that $X$ is said to be {\it essentially normal} if $X^*X-XX^*\in \mathcal{K(H)}$.

\begin{lemma}\label{lemca}
The operator $T_{\lambda(1)}$ is essentially normal and the image of $C^*(T_{\lambda(1)})$ in the Calkin algebra
is isomorphic to $C(\mathbb{T})$.
\end{lemma}
\begin{proof}
Since $T_{\lambda(1)}^*T_{\lambda(1)}-T_{\lambda(1)}T_{\lambda(1)}^*$ is the projection onto $\mathcal{H}\otimes \mathbb{C}\delta_\iota$ and $\mathcal{H}$ is finite dimensional,
$T_{\lambda(1)}$ is essentially normal.
For the second assertion it suffices to show that $\|T_x\|_{\rm e}=\| x\|$ for $x\in \lcrpr{\Z}{\alpha}{M}$.
Let $K\in \mathcal{K}(\mathbb{H}^2)$ be arbitrary.
By \cite[Poposition 3.4]{Sa}, $\rho(n)^* PxP \rho(n)$ converges to $x$ strongly.
Since $K$ is compact,
we have $x={\rm s\,\mathchar`-} \lim_{n\to \infty} \rho(n)^*(PxP-K)\rho(n)$.
By the lower-semicontinuity of operator norm, one has $\|T_x-K \| \geq \| x\|$.
Since $K$ is arbitrary, we get $\|T_x\|_{\rm e}=\| x\|$.
\end{proof}

\begin{lemma}\label{lemdi}
If $a\in \pi(M)'$ satisfies that $[T_a, T_{\lambda(1)}] \in \mathcal{K}$,
then every $\sigma$-weak cluster point of $\{\lambda(n)^*PaP \lambda(n) \}_{n\geq 0}$
falls in $\rcrpr{M'}{\alpha}{\Z}$.
\end{lemma} 
\begin{proof}
Let $b$ be a $\sigma$-weak cluster point of $\{\lambda(n)^*PaP \lambda(n) \}_{n\geq 0}$.
Then there exists a subnet $\Lambda$ of $\mathbb{N}$ such that $b=\sigma$-w-$\lim_{n\in\Lambda} \lambda(n)^* PaP \lambda(n) $.
Since $PaP \in \pi(M)'$ and $\lambda(1)$ normalizes $\pi(M)'$, one has $b\in \pi(M)'$.
Hence, it suffices to show that the $b$ commutes with $\lambda(1)$.
Since $[PaP,P\lambda(1) P] $ is compact by assumption, one has
\begin{align*}
&\phantom{aa,}\lambda(1) ( \lambda(n)^* PaP \lambda(n) )\\
&= \lambda(n)^* P\lambda(1) PaP   \lambda(n)\\
&= \lambda(n)^* PaP \lambda(1) P \lambda(n) +  \lambda(n)^* [ P\lambda(1) P,PaP ] \lambda(n) \\
&= \lambda(n)^* PaP \lambda(n)  \lambda(1) +  \lambda(n)^* [ P\lambda(1) P,PaP ] \lambda(n)
+ \lambda(n)^* PaP \lambda(1) (1 - P) \lambda(n) \\
& \xrightarrow{n \in \Lambda }    b\lambda(1) \quad \text{strongly},
\end{align*}
which implies $[b,\lambda(1)]=0$.
\end{proof}

We denote by $\| X \|_p \in [0,+\infty]$ the Schatten $p$-norm of $X\in \mathcal{B(H)}$ with $1\leq p<\infty$
and define $\|X\|_\infty:=\|X\|$, the operator norm of $X$.
Recall the fact that  the norms $\| \cdot\|_p$ are lower-semicontinuous with respect to the weak operator topology (see e.g. \cite[Proposition 2.11]{Hi}).

\begin{lemma}\label{lemlim}
Let $b \in \{\lambda(1) \}'$, $K\in\mathcal{K}$, and $1\leq p \leq \infty$ be given and set $X:=T_b + K$.
Assume that $x$ is an element in the $*$-algebra generated by $\lambda(1)$ and
there exists a constant $\delta >0$ such that $\|T_xX-T_{bx}\|_p > \delta$.
Then $\|[T_{x\lambda(n)} , X]\|_p > \delta$ and $x\lambda(n) \in \lcrpr{\Zp}{\alpha}{M}$ hold for all sufficiently large $n\in\mathbb{N}$.
\end{lemma}

\begin{proof}
Since $x$ is a polynomial of $\lambda(1)$ and $\lambda(1)^*$,
there exists $n_0\in \mathbb{N}$ such that $x\lambda(n_0)$ is in $\lcrpr{\Zp}{\alpha}{M}$.
For $n\geq n_0$ one has
\begin{align*}
T_{\lambda(n)}^* [T_xT_{\lambda(n)},X]  
&=T_{\lambda(n)}^* T_xT_{\lambda(n)}X -T_{\lambda(n)}^*(T_b + K)T_xT_{\lambda(n)}\\
&=T_xX - T_{\lambda(n) b}^* T_{x \lambda(n)} -T_{\lambda(n)}^*KT_xT_{\lambda(n)}\\
&=T_xX- T_{bx} - T_{\lambda(n)}^*KT_xT_{\lambda(n)}
\end{align*}
converges to $T_xX- T_{bx}$ strongly as $n\to \infty$, since $K$ is compact.
Thus, by the lower-semicontinuity of $\|\cdot \|_p$, there exists $n_1>n_0$ such that 
\begin{equation*}
\| [T_{x \lambda(n)},X] \|_p \geq \|T_{\lambda(n)}^* [T_xT_{\lambda(n)},X] \|_p  >  \delta
\end{equation*}
as long as $n\geq n_1$.
\end{proof}

First, we deal with the case of $p=\infty$.
Although Claim \ref{claim} below can be shown in the same way as the proof of \cite[Theorem 2]{Da},
we do give its a bit simplified proof based on the techniques used in \cite{Mu-Xi} and \cite{Xi}.
The way of our proof may be essentially known, but we could not find a suitable reference that explicitly explains such an argument.

\begin{proof}[Proof of Theorem \ref{thmf} when $p=\infty$]
Let $X$ be in $\eci(\Zp\bar{\times}_\alpha M)$.
Since $X$ is in $\eci(T(\pi(M)))=T(\pi(M)')+\mathcal{K}$,
there exists $a\in \pi(M)'$ such that $X-T_a$ is compact.
Let $\Lambda$ be a subsequence of $\mathbb{N}$ such that
the limit $b = \sigma$-w-$\lim_{n\in \Lambda}\lambda(n)^*PaP \lambda(n)$ exists.
Lemma \ref{lemdi} says that this $b$ must be in $\rcrpr{M'}{\alpha}{\Z}$.
If $T_a-T_b$ is compact, then so is $X-T_b$ too.
Since $T_a$ is in $\eci(\Zp\bar{\times}_\alpha M)$, that is, $[X,Y]$ is compact for every $Y\in \Zp\bar{\times}_\alpha M$,
the proof will be complete after establishing the next claim.
\end{proof}

\begin{claim}\label{claim}
If $T_a-T_b$ is not compact,
then there exists $Z$ in $\lcrpr{\Zp}{\alpha}{M}$  such that
$[T_Z,T_a]$ is not compact.
\end{claim}
\begin{proof}[{\rm (}$\because${\rm )}]
Set $\delta:=\| T_a-T_b \|_{\rm e}>0$.
By Lemma \ref{lemca} the image of $C^*(T_{\lambda(1)})$ in the Calkin algebra is isomorphic to $C(\mathbb{T})$.
As in the proof of \cite[Proposition 8]{Xi} we can choose $\gamma \in \mathbb{T}$ in such a way that $f(\gamma)=1$ forces
$$
\|T_{\hat{f}}T_a-T_{\hat{f}b} \|_{\rm e} = \|T_{\hat{f}}(T_a-T_b) \|_{\rm e} > \delta/2
$$
for $f\in C(\mathbb{T})$,
where $L^\infty \ni g \mapsto \hat{g} \in \{\lambda(1)\}''$ is the $*$-isomorphism in Lemma \ref{lemab}.
Let us construct a sequence $\{p_k \}_k \subset \lcrpr{\Zp}{\alpha}{M}$ such that
$\lim_{k\to\infty}\|[T_{p_k},T_a] \| \geq \delta/2$ and the sum $\sum_{k=1}^\infty p_k$ converges strongly.
Set $B_n^\gamma:=\{ \gamma' \in \mathbb{T} \,|\, |\gamma'-\gamma| < n^{-1}\}$ and
$c(f):=\|T_{\hat{f}}T_a-T_{\hat{f}b} \|$ for $f\in L^\infty(\mathbb{T})$.
Note that $c(\cdot)$ is $\sigma$-weakly lower-semicontinuous.
For $n+2 < k$ one can take $f_{n,k} \in C(\mathbb{T})$ in such a way that
$0\leq f_{n,k}\leq 1$, $f=0$ on $B^\gamma_k \cup (\mathbb{T}\setminus B^\gamma_{n})$ and $f=1$ on $B^\gamma_{n+1}\setminus B^\gamma_{k-1}$.
Note that $\{ f_{n,k} \}_k$ converges almost everywhere to a function $f_n\in C(\mathbb{T})$ with $f_n(\gamma)=1$.
Hence, by the lower-semicontinuity of $c(\cdot)$ together with Lebesgue's dominated convergence theorem,
there exists $n'>n$ such that $c(f_{n,n'})>\delta/2$.
We can inductively choose $n_1<n_1'<n_2<n_2'<\dots$ in such a way that
$g_k:=f_{n_k,n_k'}$ satisfies that $0\leq g_k \leq 1$, $\{ g_k\}_k$ have pairwise disjoint supports,
and $c(g_k)>\delta/2$.
Let $h_k$ be a trigonometric polynomial such that $\|g_k-h_k \|_\infty <2^{-k}$ and $c(h_k)>\delta/2$.
By Lemma \ref{lemlim} one can choose $m_k$ such that $\| [T_{\hat{h_k}\lambda({m_k})},T_a] \| >\delta/2$
and $\hat{h_k}\lambda({m_k})\in \lcrpr{\Zp}{\alpha}{M}$.
Set $p_k:=\hat{h_k}\lambda(m_k)$,
then $\|[T_{p_k},T_a] \| \leq 2 \| h_k\|_\infty \|a\| \leq 4\| a\|$.
Thus by passing to a subsequence if necessary,
we may and do assume that $\lim_{k\to \infty} \| [T_{p_k}, T_a]\|$ exists.
Let $\chi_k$ be the characteristic function of ${\rm supp}(g_k)$.
Then $\hat{\chi_k}, k\in \mathbb{N}$, are mutually orthogonal projections in $\{\lambda(1)\}''$ satisfying that
$\|p_k(1-\hat{\chi_k}) \| = \|h_k (1-\chi_k) \|_\infty <2^{-k}$ and
$\|p_k \hat{\chi_k}\| = \|\hat{\chi_k}p_k\hat{\chi_k}\| \leq 2$ for every $k\in \mathbb{N}$.
For any $\xi \in \mathbb{L}^2$ and $k<l$ we have
\begin{align*}
\left\| \sum_{i=k}^l p_i \xi \right\|
\leq \left\| \sum_{i=k}^l \hat{\chi_i} p_i \hat{\chi_i} \xi \right\| + \left\|\sum_{i=k}^l p_i  (1- \hat{\chi_i}) \xi \right\|
\leq 2 \left\| \sum_{i=k}^l \hat{\chi_i} \xi \right\| + 2^{-(k-1)}\|\xi \|
\end{align*}
converges to 0 as $k,l \to \infty$,
and hence the $\{p_k \}_k$ is the desired sequence.
Note that $p_k$ and $p_k^*$ converge to 0 strongly.
Since $T_a \in \eci(\Zp\bar{\times}_\alpha M)$,
we can apply \cite[Lemma 2.1]{Mu-Xi} to compact operators $[T_{p_k},T_a]$,
and obtain a subsequence $\{p_{k(i)}\}_{i=1}^\infty$ such that
$$
\sum_{i=1}^\infty [T_{p_{k(i)}},T_a] \;\;{\rm converges \;\; strongly\;\; and \;\;} \left\| \sum_{i=1}^\infty [T_{p_{k(i)}},T_a] \right\|_{\rm e} \geq \delta/2.
$$
Letting $Z:=\sum_{i=1}^\infty p_{k(i)}$ we have $Z\in \lcrpr{\Zp}{\alpha}{M}$ and $[T_Z,T_a] \notin \mathcal{K}$,
which implies the claim.
\end{proof}

%
%
We then treat the case of $p\neq \infty$.
The next lemma originates in \cite[Lemma 2.1]{Mu-Xi}.
\begin{lemma}\label{lemcpt}
Let $\mathcal{H}_1$ be a Hilbert spaces and fix $1\leq p< \infty$.
Assume that a sequence $\{ K_n \}_n \subset \mathfrak{S}_p(\mathcal{H}_1)$
satisfies the following conditions{\rm :}
\begin{itemize}
\item[(A1)]$\|K_n\|_p >2$ for every $n\in \mathbb{N}$.
\item[(A2)] $\sup_n \| K_n\| <C_1$ for some $C_1>0$.
\item[(A3)] $K_n$ and $K_n^*$ converge to $0$ strongly. 
\end{itemize}
Then there exists a subsequence $\{K_{n_k} \}_{k=1}^\infty$ such that $\sum_{k=1}^\infty K_{n_k}$ converges strongly and $\sum_{k=1}^\infty K_{n_k} \notin \mathfrak{S}_p(\mathcal{H}_1)$.
\end{lemma}
\begin{proof}
Let $\mathcal{H}_0$ be the separable Hilbert space generated by $\bigcup_{n=1}^\infty({\rm Ker}K_n)^\bot$.
Choose a CONS $\{e_n\}_{n=1}^\infty$ of $\mathcal{H}_0$, 
and let $R_n$ be the orthogonal projection onto the linear span of $e_1,\dots, e_n$.
We claim that there exist mutually orthogonal finite rank projections $\{Q_k\}_{k=1}^\infty$
and a subsequence $\{K_{n_k}\}_{k=1}^\infty$ such that
\begin{itemize}
\item[(1)]$\|Q_k K_{n_k} Q_k\|_p >1$,
\item[(2)]$\| K_{n_k} Q_k^\bot\|_p <3^{-k}$ and $\| Q_k^\bot K_{n_k} \|_p < 3^{-k}$,
\item[(3)]$\|K_{n_k} R_k\| < 2^{-k}$,
\end{itemize}
with $Q_k^\bot:=I-Q_k$.
Assume that we have chosen $Q_1,\dots,Q_k$ and $n_1,\dots,n_k$.
Put $Q:=\sum_{j=1}^k Q_j$. Since $K_n,K_n^* \to 0$ strongly and $Q$ is finite rank,
there exists $n_{k+1}>n_k$ such that
$\|K_{n_{k+1}} R_{k+1}\| < 2^{-k-1}$,
$\|K_{n_{k+1}} Q\|_p < 3^{-k-1}$,
and $\|Q K_{n_{k+1}}\|_p < 3^{-k-1}$.
Thus this $K_{n_{k+1}}$ satisfies the desired (3). Write $K:=K_{n_{k+1}}$ for short. 
We have
\begin{align*}
2
<\| K \|_p 
\leq \|Q^\bot K Q^\bot \|_p  + \|Q KQ^\bot\|_p +\|KQ\|_p
< \| Q^\bot KQ^\bot \|_p + 1,
\end{align*}
implying $\|Q^\bot KQ^\bot \|_p > 1$. 
Let $F_j \leq Q^\bot, j\in J,$ be an increasing net of finite rank projections which converges to $Q^\bot$ strongly.
Since $K$ is in $\mathfrak{S}_p$,
one has $\|F_i K F_i -Q^\bot KQ^\bot \|_p$, $\| K F_i^\bot \|_p$, and $\|F_i^\bot K \|_p$ converge to 0.
By the lower-semicontinuity of norm, we can find $j\in J$ in such a way that $Q_{k+1}:=F_j<Q^\bot $ satisfies (1) and (2) for $K=K_{n_{k+1}}$. 
Hence we can construct the desired $Q_k$ and $K_{n_k}$ by induction. 
 
Write $K_k:=K_{n_k}$ for short.
By (A2) and (2) we have
\begin{align*}
\left\| \sum_{k=1}^n K_k \right\|
\leq \left\| \sum_{k=1}^n Q_k K_k Q_k \right\|
+ \sum_{k=1}^n  \left\|  Q_k^\bot K_k Q_k \right\| + \sum_{k=1}^n \left\|  K_k Q_k^\bot \right\| 
 <  C_1 + 2.
\end{align*}
Hence $\sum_{k=1}^n K_k $ is norm bounded.
If $\xi$ in $R_n\mathcal{H}_0$ and $m\geq l \geq n$, then by (3) we have
\begin{align*}
\left\| \left(\sum_{k=1}^m K_k - \sum_{k=1}^l K_k \right)\xi \right\|
\leq \sum_{k=l+1}^m \| K_k R_n\|\,\| \xi \| 
\leq \sum_{k=l+1}^m 2^{-k} \| \xi \|  \leq 2^{-l}\| \xi\|.
\end{align*}
Since $\bigcup_{n=1}^\infty R_n\mathcal{H}_1$ is dense in $\mathcal{H}_0$ and
each $K_k$ equals to 0 on $\mathcal{H}_0^\bot$, 
$\sum_{k=1}^\infty K_k$ converges strongly. Set $X=\sum_{k=1}^\infty K_k$.
Since each $Q_n$ is finite rank, $Q_n \sum_{k=1}^m K_k Q_n$ converges to $Q_n X Q_n$ in norm. 
For each $n\in\mathbb{N}$ we have,
\begin{align*}
\|Q_n X Q_n\|_p&
= \lim_{m\to\infty}\left\| Q_n \sum_{k=1}^m K_k Q_n \right\|_p 
\geq \|Q_n K_n Q_n\|_p - \sum_{k\neq n}\|Q_n K_k Q_n\|_p \\ 
&\geq 1- \sum_{k\neq n} \|K_k(I-Q_k) \|_p
\geq 1- \sum_{k\neq n} 3^{-k}
\geq \frac{1}{2}, 
\end{align*}
which implies that $X \notin \mathfrak{S}_p$, since $Q_n, n\in \mathbb{N},$ are mutually orthogonal.
\end{proof}

%
%
We can now complete the proof of Theorem \ref{thmf}.

\begin{proof}[Proof of Theorem \ref{thmf} when $p \neq \infty$]
Let $X \in \ecp(\Zp\bar{\times}_\alpha M)$ be arbitrarily chosen.
Since $\ecp(\Zp\bar{\times}_\alpha M) \subset \eci(\Zp\bar{\times}_\alpha M)$,
there exists $b \in \rcrpr{M'}{\alpha}{\Z}$ such that
$X-T_b$ is compact.
Suppose that $X-T_b \in \mathcal{K}\setminus \mathfrak{S}_p$.
Define $c_p:L^\infty \to [0,+\infty]$ by 
$c_p(g)=\|T(\hat{g})X-T(b\hat{g}) \|_p$.
Note that $c_p$ is lower-semicontinuous with respect to the weak operator topology and $c_p(1)=+\infty$.
Suppose that for each $\gamma \in \mathbb{T}$ there exists $n\in \mathbb{N}$ such that
$c_p(f)\leq 2$ as long as $f\in C(\mathbb{T})$ satisfies that $0\leq f \leq 1$ and ${\rm supp}(f) \subset B^\gamma_n$.
Then, by the compactness of $\mathbb{T}$, there exist such $(\gamma_1,n_1),\dots,(\gamma_k,n_k) \in \mathbb{T}\times \mathbb{N}$ with $\mathbb{T}=\bigcup_{i=1}^k B^{\gamma_i}_{n_i}$.
Taking a partition of the unity $\{ \psi_i\}_{i=1}^k$ for the covering $\{ B^{\gamma_i}_{n_i} \}_{i=1}^k$,
we have $c_p(1)  \leq \sum_{i=1}^k c_p(\psi_i) \leq 2k$,
a contradiction.
Hence, we can find $\gamma \in \mathbb{T}$ such that
for every $n\in \mathbb{N}$ there exists $f_n \in C(\mathbb{T})$ such that  $c_p(f_n)>2$, $0\leq f_n \leq 1$, and ${\rm supp}(f_n) \subset B^{\gamma}_n$.
By a same approximation argument as in the case of $p=\infty$ together with Lemma \ref{lemlim},
we obtain $p_k\in \lcrpr{\Zp}{\alpha}{M}$ such that
$p_k$ and $p_k^*$ converge to 0 strongly,
$\sum_{k=1}^\infty p_k$ converges strongly,
and that $\| p_k\| \leq 2$ and $\|[T_{p_k},X]\|_p>2$ for every $k\in \mathbb{N}$.
Applying Lemma \ref{lemcpt} to $K_k=[T_{p_k},X]$ we obtain a subsequence $\{p_{k(i)}\}_i$ such that
$Z:=\sum_{i=1}^\infty p_{k(i)} \in \lcrpr{\Zp}{\alpha}{M}$ and $[T_Z, X] \notin \mathfrak{S}_p$.
This is a contradiction; hence $X-T_b$ is in $\mathfrak{S}_p$.
\end{proof}

\begin{remark}\label{remxia}\normalfont
To compute the essential fixed-points, there is a technical obstruction:
In the proof of Lemma \ref{lemdi}, it is crucial that that $\lambda(1)$ normalizes $\pi(M)'$.
In \cite{Xi} Xia computed $\efi(T(H^\infty))$ by using the finite Blaschke product $w_n$ instead of the inner function $z^n$, see \cite[Proposition 3]{Xi}.
However, unitary elements in $C^*(\lambda(1))$ do not normalize $\pi(M)'$ in general.
\end{remark}

%
%
\section{The Condition $(\star)$}
Let $M$ be a von Neumann algebra on $L^2(M)$ and $\alpha$ be a $*$-automorphism of $M$.
We say $(M,\alpha)$ satisfies the condition $(\star)$ if
\begin{equation*}
\ecp(\crpr{\Zp}{\alpha}{M}) \subset T(\rcrpr{M'}{\alpha}{\Z}) + \mathfrak{S}_p, \quad 1\leq p\leq \infty \tag{$\star$}
\end{equation*}
hold true.
Note that $(M,\alpha)$ satisfies $(\star)$ if and only if so does $(M',\alpha)$ if and only if
\begin{equation*}
\ecp(\crpr{M}{\alpha}{\Zp}) \subset T(\lcrpr{\Z}{\alpha}{M'})+ \mathfrak{S}_p, \quad 1\leq p\leq \infty
\end{equation*}
hold.
We have already seen that $(M,\alpha)$ satisfies $(\star)$ when
$M$ is either diffuse, type I$_\infty$ factor, a direct sum of them (Corollary \ref{cor}), or finite dimensional (Theorem \ref{thmf}).
In fact, we can prove the next theorem which asserts the condition $(\star)$ is satisfied in a more general case.
For a given $(M,\alpha)$ it is known and not hard to see that $M$ is decomposed uniquely as $M=M_{\rm c} \oplus M_\infty \oplus \sum_{n \geq 1}^\oplus M_n$, where $M_{\rm c}$ is diffuse, the $M_\infty$ a direct sum of infinite type I factors,
and $M_n=M_n(\mathbb{C})\otimes \ell^\infty(\mathfrak{X}_n)$, $n\geq 1$, with discrete sets $\mathfrak{X}_n$.
The uniqueness of this decomposition guarantees that $\alpha$ is also decomposed as $\alpha=\alpha_{\rm c}\oplus \alpha_\infty \oplus \sum_{n\geq 1}^\oplus \alpha_n$.
It is not difficult to see that there exists a unique automorphism $\beta_n$ of $\ell^\infty(\mathfrak{X}_n)$ and a unitary element $v_n \in M_n(\mathbb{C})\otimes \ell^\infty(\mathfrak{X}_n)$ such that $\alpha_n={\rm Ad}v_n\circ ({\rm id}\otimes \beta_n)$.
Remark that $\beta_n$ induces a unique bijection $\theta_n$ on $\mathfrak{X}_n$.
With these notation we have:

\begin{theorem}\label{thm3}
Assume that every orbit of $\theta_n$ forms a finite set.
Then $(M,\alpha)$ satisfies the condition $(\star)$.
\end{theorem}

To prove this theorem we need the following lemma.
\begin{lemma}\label{lemsum}
Let $\{(M_i,\alpha_i)\}_{ i\in I}$ be a family of von Neumann algebras and their $*$-automorphisms.
Set $M:=\sum_{i\in I}^\oplus M_i$ and $\alpha:=\sum_{i\in I}^\oplus \alpha_i$.
If every $(M_i, \alpha_i)$ satisfies the condition $(\star)$, then so does $(M,\alpha)$.
\end{lemma}
\begin{proof}
Let $X\in \ecp(\crpr{\mathbb{Z_+}}{\alpha}{M})$ be arbitrarily chosen.
Let $e_i$ be the central support of $\pi(M_i)$ in $\pi(M)$ and put $A:=\{ e_i \,|\, i\in I\}''$.
Note that $\lambda(1)$ and $P$ commute with $A$.
Then there exists $a\in \pi(M)'\subset A'$ such that $X-T_a \in \mathfrak{S}_p$.
Since $T_a$ is also in $\ecp(\crpr{\mathbb{Z_+}}{\alpha}{M})$ and $\crpr{\mathbb{Z_+}}{\alpha}{M}=\sum_{i\in I}^\oplus \crpr{\mathbb{Z_+}}{\alpha_i}{M_i}$, we have $T_{ae_i} \in \ecp(\crpr{\mathbb{Z_+}}{\alpha_i}{M_i})$.
By assumption, there exist $b_i \in \rcrpr{M'_i}{\alpha_i}{\Z}$ and $K_i\in \mathfrak{S}_p(e_iL^2(M))$
such that $T_{ae_i}=T_{b_i} + K_i$.
Set $b:=\sum_{i\in I}b_i$ and $K:=\sum_{i\in I} K_i$.
Since $b\in \sum_{i\in I}^\oplus \rcrpr{M'_i}{\alpha_i}{\Z} = \rcrpr{M'}{\alpha}{\Z}$,
it suffices to prove that  $K\in \mathfrak{S}_p$.
Since $T_{ae_i}-T_{\lambda(n)e_i}^*T_{ae_i}T_{\lambda(n)e_i} \to K_i$ strongly as $n\to \infty$,
the lower-semicontinuity of $\|\cdot\|_p$ enable us to find $n_i \in \mathbb{N}$ in such a way that
$$
\|[T_{\lambda(n_i)e_i},T_{ae_i}] \|_p \geq \| T_{ae_i}-T_{\lambda(n_i)e_i}^*T_{ae_i}T_{\lambda(n_i)e_i}\|_p > 2^{-1}\|K_i\|_p.
$$
Then $x:=\sum_{i\in I}\lambda(n_i)e_i$ falls in $\lcrpr{\Zp}{\alpha}{M}$.
Since $T_a\in \ecp(\crpr{\mathbb{Z_+}}{\alpha}{M})$, one has $[T_x,T_a] = \sum_{i\in I} [T_{\lambda(n_i)e_i},T_{ae_i}] \in \mathfrak{S}_p$.
Hence, it follows from the inequality above that $K\in \mathfrak{S}_p$.
\end{proof}

\begin{proof}[Proof of Theorem \ref{thm3}]
By Corollary \ref{cor} and Lemma \ref{lemsum} we may and do assume that $M=M_n$ and $\alpha={\rm id}\otimes \beta_n$.
Decompose $\mathfrak{X}_n$ into the disjoint $\theta_n$-orbits $\mathfrak{X}_{n,j}$, $j\in J_n$.
Set $M_{n,j}:=M_n(\mathbb{C})\otimes \ell^\infty(\mathfrak{X}_{n,j})$, which sits inside $M_n=M_n(\mathbb{C})\otimes \ell^\infty(\mathfrak{X}_n)$ naturally.
Then $M_n=\sum_{j\in J_n}^\oplus M_{n,j}$, and clearly $\alpha(M_{n,j})=M_{n,j}$ holds for every $j \in J_n$.
Consequently, one has $(M_n, \alpha_n)=\sum_{j \in J_n}^\oplus (M_{n,j},\alpha |_{M_{n,j}})$.
By assumption each $\mathfrak{X}_{n,j}$ is a finite set, and hence $M_{n,j}$ is finite dimensional.
Therefore, the desired assertion follows from Theorem \ref{thmf} thanks to Lemma \ref{lemsum}.
\end{proof}

Here a question naturally arises.

\begin{question}\label{ques}
Let $\sigma$ be the $*$-automorphism on $\ell^\infty(\Z)$ induced from the translation $n\in \Z \mapsto n+1 \in \Z$.
Does $(\ell^\infty(\Z),\sigma )$ satisfy the condition $(\star)$?
\end{question}

In fact, the positive answer to the question enables us to get rid of the assumption from Theorem \ref{thm3} as follows.
We use the notation in the proof of Theorem \ref{thm3}.
Thanks to Lemma \ref{lemsum} and Theorem \ref{thm3},
it suffices to prove that for every infinite $\theta_n$-orbit $\mathfrak{X}_{n,j}$, $(M_{n,j},\alpha|_{M_{n,j}})$ satisfies the condition $(\star)$. 
Then $(M_{n,j},\alpha|_{M_{n,j}})$ can be identified with $(M_n(\mathbb{C})\otimes \ell^\infty(\Z),{\rm Ad}v_{n,j} \circ {\rm id}\otimes \sigma)$ for a unitary element $v_{n,j} \in M_n(\mathbb{C})\otimes \ell^\infty(\mathfrak{X}_{n,j})$,
and hence we may assume that $(M,\alpha)=(M_n(\mathbb{C})\otimes \ell^\infty(\Z),{\rm id}\otimes \sigma)$.
Write $\ell^\infty:=\ell^\infty(\Z)$ and $\ell^2:=\ell^2(\Z)$ for simplicity.
The standard form of $M_n(\mathbb{C})\otimes \ell^\infty$ becomes $M_n(\mathbb{C})\otimes \mathbb{C}1 \otimes \ell^\infty$ on $\mathbb{C}^n\otimes \mathbb{C}^n \otimes \ell^2$.
Hence, $\lcrpr{\Z}{{\rm id}\otimes \sigma}{(M_n(\mathbb{C})\otimes \ell^\infty)}$ and $\rcrpr{(M_n(\mathbb{C})\otimes \ell^\infty)'}{{\rm id}\otimes \sigma}{\Z}$
become $M_n(\mathbb{C})\otimes \mathbb{C}1 \otimes (\lcrpr{\Z}{\sigma}{\ell^\infty})$ and $\mathbb{C}1\otimes M_n(\mathbb{C}) \otimes (\rcrpr{\ell^\infty}{\sigma}{\Z})$,
respectively.
It is easily seen that
$$
\ecp(M_n(\mathbb{C})\otimes \mathbb{C}1 \otimes (\crpr{\Zp}{\sigma}{\ell^\infty})) \subset \mathbb{C}1\otimes M_n(\mathbb{C}) \otimes \ecp(\crpr{\Zp}{\sigma}{\ell^\infty}) + \mathfrak{S}_p.
$$
Therefore, if Question \ref{ques} had an affirmative answer, then $(M,\alpha)$ would satisfy the condition $(\star)$
since
\begin{align*}
\ecp(\crpr{\mathbb{Z_+}}{\alpha}{M})
&=  \ecp(  M_n(\mathbb{C})\otimes \mathbb{C}1 \otimes (\crpr{\Zp}{\sigma}{\ell^\infty}) ) \\
&\subset \mathbb{C}1\otimes M_n(\mathbb{C}) \otimes \ecp(\crpr{\Zp}{\sigma}{\ell^\infty}) + \mathfrak{S}_p \\
&\subset \mathbb{C}1\otimes M_n(\mathbb{C})\otimes T(\rcrpr{\ell^\infty}{\sigma}{\Z}) + \mathfrak{S}_p\\
&= T(\rcrpr{M'}{\alpha}{\Z}) + \mathfrak{S}_p.
\end{align*}
Finally, we should remark that the canonical implementing unitary operator of $\sigma$ is nothing but the bilateral shift on $\ell^2(\Z)$ with respect to the standard basis.
Hence Question \ref{ques} seems operator theoretic rather than operator algebraic.

\section*{Acknowledgment}
The author wishes to express his sincere gratitude to his supervisor, Professor Yoshimichi Ueda for his passionate guidance and continuous encouragement.


\begin{thebibliography}{99}
\bibitem{Br-Ha} A.~Brown and P.R.~Halmos, Algebraic properties of Toeplitz operators. {\it J. Reine Angew. Math.} {\bf 213} 1963/1964 89--102.

\bibitem{Da} K.R.~Davidson, On operators commuting with Toeplitz operators modulo the compact operators, {\it J. Funct. Anal.} {\bf 24} (1977) 291--302. 

\bibitem{Do} R.G.~Douglas, Banach Algebra Techniques in Operator Theory, second ed., Grad. Texts in Math. 179, {\it Springer-Verlag, New York}, 1998.

\bibitem{Ha} U.~Haagerup, The standard form of von Neumann algebras, {\it Math. Scand.} {\bf 37} (1975), no. 2, 271--283.

\bibitem{He-Ro} E.~Hewiit and K. A.~Ross, Abstract Harmonic Analysis, Vol. I, second ed., Grundlehren der Mathematischen Wissensrhaften [Fundamental Principles of Mathematical Sciences], vol. 115, {\it Springer-Verlag, Berlin}, 1979.

\bibitem{Hi} F.~Hiai, Log-majorizations and norm inequalities for exponential operators, in {\it Linear Operators. Banach Center Publ}, 38, 119--181 (1997)

\bibitem{Ho} T.S.~Hoover, Derivations, homomorphisms, and operator ideals, {\it Proc. Amer. Math. Soc.} {\bf 62} (1977), no. 2, 293--298.

\bibitem{Jo-Pa} B.E.~Johnson and S.K.~Parrott, Operators commuting with a von Neumann algebra modulo the set of compact operators, {\it J. Funct. Anal.} {\bf 11} (1972), 39--61.

\bibitem{Ka} E.~Kakariadis, Semicrossed products and reflexivity. {\it J. Operator Theory} {\bf 67} (2012), no. 2, 379--395.

\bibitem{Lo-Mu}R.I.~Loebl and P.S.~Muhly, Analyticity and flows in von Neumann algebras. {\it J. Funct. Anal.} {\bf 29} (1978), no. 2, 214--252.

\bibitem{Mc-Mu-Sa} M.~McAsey, P.S.~Muhly, and K-S.~Saito, Nonselfadjoint crossed products (invariant subspaces and maximality). {\it Trans. Amer. Math. Soc.} {\bf 248} (1979), no. 2, 381--409. 

\bibitem{Mu-Xi} P.S.~Muhly and J.~Xia, On automorphisms of the Toeplitz algebra, {\it Amer. J. Math.} {\bf 122} (2000), no. 6, 1121--1138.

\bibitem{Po} S.~Popa, The commutant modulo the set of compact operators of a von Neumann algebras, {\it J. Funct.
Anal.} {\bf 712} (1987), 393--408.

\bibitem{Sa} K-S.~Saito, Toeplitz operators associated with analytic crossed products, {\it Math. Proc. Cambridge Philos. Soc.} {\bf 108} (1990), no. 3, 539--549.

\bibitem{Ta} M.~Takesaki, Theory of Operator Algebras. II. Encyclopaedia of Mathematical Sciences, 125. Operator Algebras and Non-commutative Geometry, 6. {\it Springer-Verlag, Berlin}, 2003.

\bibitem{Xi} J.~Xia, A characterization of compact perturbations of Toeplitz operators. {\it Trans. Amer. Math. Soc.} {\bf 361} (2009), no. 10, 5163--5175.
\end{thebibliography}
\end{document}